\documentclass[11pt,letterpaper]{amsart}

\usepackage[T1]{fontenc}
\usepackage{amsmath,amssymb,amsthm}
\usepackage[margin=1in]{geometry}
\usepackage[colorlinks=true, linkcolor=blue, citecolor=blue]{hyperref}
\usepackage{bbm}

\DeclareMathOperator{\pf}{Pf}
\DeclareMathOperator{\crossings}{cr}
\DeclareMathOperator{\wt}{wt}
\DeclareMathOperator{\swt}{swt}

\newcommand{\bbR}{\mathbb{R}}

\newcommand{\bone}{\mathbbm{1}}
\newcommand{\cI}{\mathcal{I}}

\newtheoremstyle{customtheorem}{}{}
{}
{}{\bf}{}{ }{\thmname{#1}\thmnumber{ #2}\textnormal{\thmnote{ (#3)}}}
\theoremstyle{customtheorem}
\newtheorem{theorem}{Theorem}[section]
\newtheorem{prop}[theorem]{Proposition}

\newtheorem{lemma}[theorem]{Lemma}

\newtheorem{cor}{Corollary}[section]

\newtheorem{conj}[theorem]{Conjecture}

\newtheorem{example}[theorem]{Example}
\newtheorem*{rmk}{Remark}

\title[Pfaffian Positivity for Convex Kernels]{Pfaffian Positivity for Convex Kernels and Optimal Matching Bounds}
\author{William Johnson}
\thanks{The author is currently enrolled as a non-degree student at the University of Maryland, College Park. This research was conducted independently.}
\email{wlj1@umd.edu}
\date{October 2026} 
\subjclass[2020]{Primary 15A15; Secondary 05A20, 15B48, 70F10, 05C70, 26A51} 
\keywords{Pfaffian positivity, convex kernels, skew-symmetric matrices, matching bounds, Albouy--Moeckel conjecture, central configurations}

\begin{document}

\begin{abstract}
We show that for all positive, nonincreasing, convex functions $f : (0, \infty) \to (0,\infty)$ and for any collection of $n$ points $x_i \in \mathbb{R}$ with $x_1 < x_2 < \ldots < x_n$ and $n$ even, the Pfaffian of the skew matrix defined by $f(x_j - x_i)$ above the diagonal is positive.
More generally, we establish coefficientwise nonnegativity for Pfaffians arising from interval decompositions, which also yields quantitative lower bounds for singly and doubly bordered Pfaffians.
Choosing $f(r) = r^{-2}$ proves the three nonvanishing assertions in the conjecture of Albouy and Moeckel.
Further, we show that the rainbow bound (the product over the nested matching $\{1,n\}, \{2,n-1\}, \ldots \{n/2, n/2 + 1\}$) is the unique single-matching product that bounds $\mathrm{Pf}(A_f(x))$ below uniformly over $x$ and $f$ positive, nonincreasing, convex.
\end{abstract}

\maketitle

\section{Introduction}
\label{sec:introduction}

The Pfaffian of a skew-symmetric matrix is an alternating sum over perfect matchings, so its positivity is not generally apparent from the signs of its entries. 
In this paper, we establish positivity for a class of skew-symmetric matrices arising from ordered points and convex kernels. 
More generally, we prove a coefficientwise positivity result for matrices constructed from intervals, providing a combinatorial-algebraic approach to Pfaffian positivity.

For $f:(0,\infty)\to(0,\infty)$ positive, nonincreasing, convex, and $x_1<x_2<\ldots<x_n$ in $\mathbb R$, we consider the skew-symmetric matrix $A_f(x)$ defined by

\[
(A_f(x))_{i,j}=
\begin{cases}
f(x_j-x_i), & i<j,\\
0, & i=j,\\
-f(x_i-x_j), & i>j.
\end{cases}
\]

We prove that $\pf(A_f(x))$ is strictly positive for every even $n$. 
In fact, we establish the quantitative lower bound
\[\pf(A_f(x)) \geq f(x_n-x_1)^{n/2}>0
\]
which we call the \emph{diameter bound}.

Our proof relies on a more general algebraic result. 
We introduce an indeterminate $a_J$ for each subinterval $J$ of $[n]$ with at least two elements and construct a skew-symmetric matrix $A_n(a)$ whose entries are sums of these indeterminates. 
We show that $\pf(A_n(a))$ is a polynomial with nonnegative integer coefficients. 
Convexity and monotonicity of $f$ then imply that $A_f(x)$ is obtained by specializing these indeterminates to nonnegative values. 
This separates the combinatorial structure of the Pfaffian from the properties of the kernel.

We further strengthen the diameter bound using the \emph{rainbow matching}, which pairs $1$ with $n$, $2$ with $n-1$, and so on.

\begin{theorem}[Rainbow Bound]
  \label{thm:rainbow-bound}
  Let $n\geq2$ be even, and let $A$ be a real skew-symmetric $n\times n$ matrix with nonnegative interval coefficients (defined in \autoref{sec:interval-decomposition}). 
  Then
    \[\pf(A) \geq \prod_{k=1}^{n/2} A_{k,n+1-k} \geq A_{1,n}^{n/2}.\]
  In particular, if $f:(0,\infty)\to(0,\infty)$ is positive, nonincreasing, and convex, then
    \[\pf(A_f(x))\geq\prod_{k=1}^{n/2}f(x_{n+1-k}-x_k).\]
\end{theorem} 

Moreover, this bound is optimal among lower bounds given by the product associated with a single perfect matching.

\begin{theorem}[Optimality of the Rainbow Bound]
  \label{thm:rainbow-bound-optimal}
  Let $n\geq4$ be even, and let

\[
R=\{\{k,n+1-k\}:1\leq k\leq n/2\}
\]

be the rainbow matching. For every perfect matching $M\neq R$, the infimum over positive, nonincreasing, convex $f$  and $x_1 < \ldots < x_n$ vanishes.
\[
\inf_{f, x}
    \frac{\pf(A_f(x))} {\displaystyle\prod_{\substack{\{i,j\}\in M\\i<j}} f(x_j-x_i)} = 0.
    \]
\end{theorem}

The inverse problem for collinear central configurations was studied by Moulton \cite{Moulton1910}, who introduced the even-order Pfaffian associated with this problem.
Buchanan \cite{buchanan1909} claimed a proof that Moulton's Pfaffian never vanishes, but Albouy and Moeckel \cite[p.~91]{albouyMoeckel2000} note a gap in the argument.
Albouy and Moeckel revisited the problem without fixing the center of mass in advance. 
This places the even and odd cases on equal footing: in both, a configuration determines a two-parameter family of mass vectors making it central, provided certain Pfaffians are nonzero. 
They conjectured that these Pfaffians never vanish.
They proved nonvanishing for $n\leq4$ under certain assumptions and gave computer-assisted proofs for $n=5,6$ in the Newtonian case. 
Specifically $n\le 4$, Albouy and Moeckel prove nonvanishing under hypotheses on $\varphi(s)=f(\sqrt s)/\sqrt s$: strict monotonicity for $K_3^L$, power laws for $K_4$, and strict convexity for $K_4^{xL}$. 
The first and third are implied by our hypotheses.
They also showed that for all $n$ and almost all configurations, the Pfaffians are nonvanishing for a large class of analytic potentials with sufficiently strong singularities.

In 2014, Xie proved the Newtonian $n=6$ case in \cite{Xie2014An} with an analytic method replacing the computer-assisted method from \cite{albouyMoeckel2000}. 
Xie established an equivalence between nonvanishing Pfaffians for even numbers of bodies and bordered Pfaffians for odd numbers of bodies in the Newtonian case in \cite{xie2022inverse}.
Ferrario obtained further results in \cite{Ferrario2020, Ferrario2020Pfaffians,ferrario2026pfaffians} eventually proving Pfaffian positivity for even $n\leq14$ and positive, decreasing, log-convex kernels.

We extend these results by providing a uniform combinatorial-algebraic proof of positivity for all even $n$ and the corresponding bordered Pfaffians for odd $n$, under the weaker than log-convexity assumptions of positivity, monotonicity, and convexity. 
\autoref{lem:alternating-prefix-sum} also yields quantitative lower bounds for singly and doubly bordered Pfaffians. 
Together with the ordinary Pfaffian bound, these establish the nonvanishing of all three Pfaffians conjectured by Albouy and Moeckel for the Newtonian interaction, and extend the result to positive, nonincreasing, convex kernels.
The nonvanishing of these Pfaffians ensures that, for every noncollision collinear configuration, the space of real mass vectors making it central has dimension two.

\section{Interval Decomposition}
\label{sec:interval-decomposition}
Our proof of Pfaffian positivity separates the combinatorial structure of the Pfaffian from the analytic properties of the convex kernel.
We begin by introducing an indeterminate for each subinterval of $[n]$ and constructing a skew-symmetric matrix whose entries are sums of these indeterminates.
M\"obius inversion shows that every skew-symmetric matrix arises by specializing the indeterminates.
We will then study the Pfaffian as a polynomial in the indeterminates.

Let $\cI_n = \{[p,q] \mid 1 \leq p < q \leq n\}$ be the set of all subintervals of $[n]$ with at least two elements.
For an $n \times n$ matrix $A$, define its interval coefficients by
\[ a_{[p,q]}(A) = A_{p,q} - A_{p-1,q} - A_{p,q+1} + A_{p-1,q+1}, \]
with the convention that $A_{i,j}=0$ whenever either index lies outside $[n]$.
These coefficients arise from M\"obius inversion on the poset $\cI_n$ ordered by inclusion \cite[Section~3.7]{stanley2012enumerative}.
We say $A$ has nonnegative interval coefficients if $a_J(A) \geq 0$ for all $J \in \cI_n$.

Now let $a=(a_J)_{J\in\cI_n}$ be independent commuting indeterminates. Define the skew-symmetric matrix $A_n(a)$ by
\[ (A_n(a))_{ij} =\sum_{\substack{J\in\cI_n\\i,j\in J}}a_J \qquad (i<j), \]
with the remaining entries determined by skew-symmetry.

\begin{lemma}[Interval Decomposition]
  \label{lem:interval-decomposition}
  For $A$ an $n \times n$ matrix and $i < j$
  \[A_{i,j} = \sum_{[i,j] \subseteq [p,q] \in \cI_{n}} a_{[p,q]}(A).\]
  So $a_{[1,n]}(A) = A_{1,n}$; when $A$ is skew-symmetric, $A_n(a(A)) = A$.
\end{lemma}

\begin{proof}
  By the definition of the interval coefficients,
  \begin{align*}
    \sum_{[i,j]\subseteq[p,q]} a_{[p,q]}(A)
    &= \sum_{q=j}^n \sum_{p=1}^i
       (A_{p,q}-A_{p-1,q}-A_{p,q+1}+A_{p-1,q+1})\\
    &= \sum_{q=j}^n (A_{i,q}-A_{i,q+1})\\
    &= A_{i,j},
  \end{align*}
  where both sums telescope.

  The identity $a_{[1,n]}(A)=A_{1,n}$ follows directly from the definition. 
  If $A$ is skew-symmetric, then $A_n(a(A))=A$.
\end{proof}

The interval decomposition allows us to express any skew-symmetric matrix in terms of independent indeterminates indexed by intervals. 
In the next section, we apply this decomposition to the Pfaffian and prove that the resulting polynomial has nonnegative integer coefficients.

\section{Pfaffian Positivity}
\label{sec:pfaffian-positivity}

We use the perfect-matching expansion of the Pfaffian to establish coefficientwise nonnegativity of the Pfaffian of the interval matrix constructed in the previous section. 
The key observation is that adjacent vertices with identical connections to the remaining vertices can be removed without changing the Pfaffian, apart from a multiplicative factor.
This is combined with the convexity of $f$ to show that $\pf(A_f(x)) > 0$.

Let $n$ be a nonnegative even integer and $A$ an $n\times n$
skew-symmetric matrix over a commutative ring. 
Let $\Pi_n$ denote the set of perfect matchings of $[n]$.
We regard each matching $M\in\Pi_n$ as a set of pairs $\{i,j\}$ with $i<j$.
Two pairs $\{i,j\}$ and $\{k,l\}$ cross if $i<k<j<l$ or $k<i<l<j$.
Let $\crossings(M)$ denote the number of crossing pairs in $M$.

Define the weight and signed weight of a matching $M$ by
\[\wt(M)=\prod_{\{i,j\}\in M}A_{i,j} \qquad \swt(M)=(-1)^{\crossings(M)}\wt(M). \]
The Pfaffian of $A$ is
\[ \pf(A)=\sum_{M\in\Pi_n}\swt(M), \]
with the convention $\pf(\emptyset)=1$.
For $S\subseteq[n]$, let $A(S)$ denote the principal submatrix of $A$ indexed by $S$, with indices in increasing order.

Finally, for the interval matrix $A_n(a)$ defined in the previous section, let
\[ \pf_n(a)=\pf(A_n(a)). \]
By \autoref{lem:interval-decomposition}, 
\[ \pf(A)=\pf_n(a(A)) \]
for every skew-symmetric $n\times n$ matrix $A$.

\begin{lemma}[Twin Rows]
  \label{lem:twin-rows}
  Let $A$ be an $ n \times n$ skew-symmetric matrix over a commutative ring, where $n \geq 2$ is an even integer.
  Suppose for some $1 \leq i < n$, \[A_{i,j} = A_{i+1,j} \text{ for } j \not \in \{i, i + 1\}.\]
  Then with $S_i = [n] \setminus \{i, i+1\}$ we get \[\pf (A) = A_{i,i+1} \pf\left(A(S_i) \right).\]
\end{lemma}

\begin{proof}
  We will consider all the matchings $M \in \Pi_n$.
  First, suppose that $\{i,i+1\} \in M$.
  Then $\{i, i+1\}$ does not cross any other edges in $M$ because there are no points between $i$ and $i+1$.
  So $(-1)^{\crossings(M)} = (-1)^{\crossings(M \setminus \{\{i, i+1\}\})}$.
  Summing over all such $M$ in the expansion of $\pf(A)$ we get $A_{i,i+1}\pf(A(S_i))$.

  Otherwise, suppose that $\{i,i+1\} \not \in M$.
  Then there are $j, j'$ with $\{i,j\}, \{i+1, j'\} \in M$ with $j \neq j'$ and $j,j' \not \in \{i,i+1\}$.
  Let $M' = M \setminus \{\{i,j\},\{i+1,j'\}\} \cup \{\{i,j'\}, \{i+1, j\}\}$, i.e. the matching from swapping the other endpoints of the edges that $i,i+1$ are members of.
  Note then that \[M \mapsto M'\] is an involution on $\{M \in \Pi_n \mid \{i,i+1\} \not \in M\}$ with no fixed points and which preserves weight since $j > i+1$ gives $A_{i,j} = A_{i+1,j}$ and for $j < i$ skew symmetry gives $A_{j,i} = A_{j,i+1}$.

  The result will follow from $\swt(M) = - \swt(M')$.
  There are two cases for the edges in $M,M'$.
  For edges $\{s,t\}$ besides $\{i,j\},\{i+1,j'\}$, there are no points between $i,i+1$ and $s,t \not \in \{i,i+1\}$.
  So $\{s,t\}$ crosses $\{i,j\}$ if and only if it crosses $\{i+1,j\}$ and similarly for $\{i,j'\}$.
  Hence the number of crossings between $\{s,t\}$ and the two edges containing $i,i+1$ are the same in both $M$ and $M'$.

  Now consider, for the edges $\{i,j\},\{i+1,j'\}$ and $\{i,j'\}, \{i+1,j\}$ exactly one of the two configurations of the swapped edges has a crossing.
  If $j,j' > i+1$, then $\{i,j\}$ and $\{i+1,j'\}$ cross if and only if $j < j'$ while $\{i,j'\}$ and $\{i+1,j\}$ cross if and only if $j > j'$.
  The case $j,j' < i$ is symmetric.
  Similarly, if $j < i < i+1 < j'$, then $\{j,i\}$ and $\{i+1 , j'\}$ do not cross, while $\{j,i+1\}$ and $\{i,j'\}$ do.
  And again, $j' < i < i+1 < j$ is symmetric.

  So $M$ and $M'$ have opposite crossing parity, so $\swt(M) + \swt(M') = 0$ and all of the relevant terms cancel from $\pf(A)$.
  Hence we find \[\pf (A) = A_{i,i+1} \pf\left(A(S_i) \right).\]
\end{proof}

We now prove coefficientwise nonnegativity by induction on $n$. The argument relies on the observation that a sufficiently small collection of intervals either leaves a vertex uncovered or contains two consecutive vertices belonging to exactly the same intervals. In the latter case, \autoref{lem:twin-rows} reduces the Pfaffian to one of smaller order, while preserving the interval decomposition.

\begin{theorem}[Coefficient Positivity]
  \label{thm:coefficientwise-positive}
  For every even $n \geq 0$, $\pf_n(a)$ is a polynomial with nonnegative integer coefficients in the indeterminates $a=(a_J)_{J\in\cI_n}$.
  It is homogeneous of degree $n/2$, and for $n\geq 2$ the coefficient of $a_{[1,n]}^{n/2}$ is $1$.
\end{theorem}

\begin{proof}
  Each entry of $A_n(a)$ is linear in the indeterminates $a_J$. 
  Since every perfect matching has $n/2$ edges, the matching expansion shows that $\pf_n(a)$ is a homogeneous polynomial of degree $n/2$ with integer coefficients.

  We prove coefficientwise nonnegativity by induction on even $n$. 
  The base cases are
  \[\pf_0(a)=1, \qquad \pf_2(a)=a_{[1,2]}.\]

  Let $n\geq 4$ and assume the result for $n-2$.
  Fix an arbitrary monomial
  \[
    a^\mu=\prod_{J\in K}a_J^{\mu_J}
  \]
  of degree $n/2$, where $K$ is its support.
  In particular, $|K|\leq n/2$.

  Set
  \[A_n^K(a) = \left.A_n(a)\right|_{a_J=0,\ J\notin K}, \qquad \pf_n^K(a)=\pf(A_n^K(a)).\]
  Specializing variables outside $K$ to zero does not change the coefficient of $a^\mu$. 
  Thus it suffices to show that $\pf_n^K(a)$ has nonnegative coefficients.

  Suppose first that some $v\in[n]$ belongs to no interval in $K$. 
  Then row $v$ of $A_n^K(a)$ is zero, so $\pf_n^K(a)=0$.

  Otherwise, every element of $[n]$ belongs to at least one interval in $K$. 
  We claim that there are two consecutive elements belonging to exactly the same intervals in $K$.

  Say that an interval $J=[p,q]$ \emph{cuts} the gap $g\in[n-1]$ if exactly one of $g,g+1$ belongs to $J$.
  Each interval cuts at most two gaps, one at each endpoint.

  Since every vertex is covered, there is an interval $J_1\in K$ containing $1$ and an interval $J_n\in K$ containing $n$. 
  Each cuts at most one gap. 
  If $J_1=J_n$, then $J_1=[1,n]$ and cuts no gaps. 
  In either case, the total number of gaps cut by intervals in $K$ is at most
  \[
    2|K|-2\leq n-2.
  \]
  As there are $n-1$ gaps, some gap $g$ is not cut by any interval in $K$.

  Consequently,
  \[g\in J \iff g+1\in J \qquad (J\in K).\]
  Hence rows $g$ and $g+1$ of $A_n^K(a)$ agree outside columns $g,g+1$. 
  By the Twin Rows lemma,
  \[\pf_n^K(a) = s\,\pf\bigl(A_n^K(a)(S_g)\bigr),\]
  where
  \[S_g = [n]\setminus\{g,g+1\}, \qquad s=\sum_{\substack{J\in K\\g\in J}}a_J.\]

  It remains to express the smaller matrix in the same interval form.
  Let
  $\phi:S_g\to[n-2]$ be the unique increasing bijection. 
  For each $J\in K$, let
  \[T_J = \phi(J\cap S_g).\]
  Since $g$ is uncut, $J$ contains either both $g,g+1$ or neither. 
  Thus $T_J$ is an interval of $[n-2]$, possibly empty or a singleton.

  For each $L\in\cI_{n-2}$, define
  \[b_L = \sum_{\substack{J\in K\\T_J=L}}a_J.\]
  Empty and singleton images contribute no entries to the smaller matrix.

  Indeed, for $i<j$ in $S_g$,
  \begin{align*}
    A_n^K(a)_{i,j}
    &=\sum_{\substack{J\in K\\i,j\in J}}a_J\\
    &=\sum_{\substack{J\in K\\
             \phi(i),\phi(j)\in T_J}}a_J\\
    &=\sum_{\substack{L\in\cI_{n-2}\\
             \phi(i),\phi(j)\in L}}b_L.
  \end{align*}
  Therefore, after relabeling by $\phi$,
  \[A_n^K(a)(S_g) = A_{n-2}(b),\]
  and hence
  \[\pf_n^K(a) = s\,\pf_{n-2}(b).\]

  By the induction hypothesis, $\pf_{n-2}$ has nonnegative integer coefficients as a polynomial, so $\pf_{n-2}(b)$ has nonnegative integer coefficients.
  Each $b_L$, as well as $s$, is a sum of indeterminates $a_J$ with nonnegative coefficients.
  Thus $\pf_n^K(a)$ also has nonnegative integer coefficients.

  In particular, the coefficient of $a^\mu$ in $\pf_n(a)$ is nonnegative. 
  Since $a^\mu$ was arbitrary, $\pf_n(a)$ is coefficientwise nonnegative.

  Finally, for $n\geq2$, specialize $a_J=0$ for all $J\neq[1,n]$. Then
  \[A_n(a) = a_{[1,n]}\tilde{\bone}_n,\]
  where $\tilde{\bone}_n$ is the skew-symmetric matrix with every entry above the diagonal equal to $1$.
  Repeated application of \autoref{lem:twin-rows} gives $\pf(\tilde{\bone}_n)=1$, starting from $\pf(\emptyset)=1$. 
  Therefore
  \[\left.\pf_n(a)\right|_{a_J=0,\ J\neq[1,n]} = a_{[1,n]}^{n/2},\]
  proving the final assertion.
\end{proof}

As an example, consider the skew-symmetric $4 \times 4$ matrix defined in \autoref{sec:interval-decomposition}.

\begin{example}
  
Although the perfect-matching expansion contains a negative term, the interval decomposition allows the Pfaffian to be rewritten as a polynomial with nonnegative coefficients. 
In the four-point case, the difference between the Pfaffian and the rainbow matching product also has an explicit nonnegative expansion shown below.
\[ A_4(a) = \begin{pmatrix}
  0 & a_{[1,2]} + a_{[1,3]} + a_{[1,4]} & a_{[1,3]} + a_{[1,4]} & a_{[1,4]} \\
  * & 0 & a_{[1,3]} + a_{[1,4]} + a_{[2,3]} + a_{[2,4]} & a_{[1,4]} + a_{[2,4]} \\
  * & * & 0 & a_{[1,4]} + a_{[2,4]} + a_{[3,4]} \\
  * & * & * & 0
\end{pmatrix}
\]
Calculating the Pfaffian of this matrix we find 
\[\pf(A_4(a)) = a_{[1,4]}^2 + a_{[1,2]}(a_{[2,4]} + a_{[3,4]}) + a_{[1,3]} a_{[3,4]} + a_{[1,4]} (a_{[1,2]} + a_{[1,3]} + a_{[2,3]} + a_{[2,4]} + a_{[3,4]}).\]
The first term of this provides the bound in \autoref{thm:pfaffian-positive-diameter}.
Additionally, a slight rearrangement gives 
\[\pf(A_4(a)) - A_4(a)_{1,4} A_4(a)_{2,3} = a_{[1,2]}A_4(a)_{3,4} + a_{[3,4]}A_4(a)_{1,3},\]
which shows the rainbow bound for $n=4$.

\end{example}
The coefficient-positivity and endpoint-inequality results extend these phenomena to arbitrary even order.

Having established coefficientwise nonnegativity of $\pf_n(a)$, we now use the convexity of $f$ to show that the interval coefficients of $A_f(x)$ are nonnegative. 
Together with the strict positivity of $a_{[1,n]}(A_f(x))$, this will give our desired Pfaffian positivity.

\begin{lemma}[Nonnegativity of Interval Coefficients]
  \label{nonnegative-interval-coefficient}
  Let $n \geq 2$, $f : (0, \infty) \to (0, \infty)$ be positive, nonincreasing, and convex and $x_1 < x_2 < \ldots < x_{n}$.
  Then the interval coefficients satisfy $a_J(A_f(x)) \geq 0$ for all $J \in \cI_n$ and $a_{[1,n]}(A_f(x)) = f(x_n - x_1) > 0$.
\end{lemma}

\begin{proof}
  Write $A = A_f(x)$.
  \begin{itemize}
  \item If $p = 1$ and $q < n$ then $a_{[1,q]}(A) = f(x_q - x_1) - f(x_{q+1} - x_1) \geq 0$ because $f$ is nonincreasing.
  \item If $1 < p$ and $q = n$ then $a_{[p,n]}(A) = f(x_n - x_p) - f(x_n - x_{p-1}) \geq 0$ also because $f$ is nonincreasing.
  \item If $p = 1$ and $q = n$ then $a_{[1,n]}(A) = f(x_n - x_1) > 0$.
  \item Finally, if $1 < p < q < n$, let $r = x_q - x_p$, $s = x_p - x_{p-1}$, and $t = x_{q+1} - x_q$, so $r,s,t > 0$ and \[a_{[p,q]}(A) = f(r) + f(r+s+t) - f(r+s) - f(r+t)\]
    Let $\lambda = \frac{t}{s+t} \in (0,1)$ then $r + s = \lambda r + (1 - \lambda)(r+s+t)$ and $r+t = (1 - \lambda)r + \lambda ( r+s+t)$.
    Convexity of $f$ gives us $f(r+s) \leq \lambda f(r) + (1-\lambda)f(r+s+t)$ and $f(r+t) \leq (1- \lambda)f(r) + \lambda f(r+s+t)$.
    Adding these together, we find $a_{[p,q]}(A) \geq 0$.
  \end{itemize}
  
\end{proof}

\begin{theorem}[Pfaffian Diameter Bound]
  \label{thm:pfaffian-positive-diameter}
  Let $f: ( 0, \infty) \to (0, \infty)$ be positive, nonincreasing, and convex, and $x_1 < x_2 < \ldots < x_{n}$, with $n \geq 2$ an even integer. Then 
  \[\pf \left(A_f(x) \right) \geq f(x_n - x_1)^{n/2} > 0.\]
\end{theorem}

\begin{proof}
  Let $A = A_f(x)$ and $a = (a_J(A))_{J \in \cI_n}$.
  By \autoref{nonnegative-interval-coefficient},
  $a_J(A) \geq 0$ for every $J \in \cI_n$.
  Hence, by \autoref{lem:interval-decomposition}
  and \autoref{thm:coefficientwise-positive},
  \[\pf(A) = \pf_n(a) \geq a_{[1,n]}(A)^{n/2} = f(x_n-x_1)^{n/2} > 0.\]
\end{proof}

\section{Rainbow Matching Bound}
\label{sec:rainbow-bound}

The diameter bound uses only one term of the polynomial $\pf_n(a)$. 
We now obtain a stronger bound by comparing the Pfaffian with the product associated to the rainbow matching. 
The argument uses coefficientwise positivity to control certain alternating sums of principal Pfaffians, which in turn yield an inequality relating a Pfaffian to the principal Pfaffian obtained by deleting its two endpoints.

Let $n\geq1$ be odd and $A$ an $n\times n$ skew-symmetric matrix. 
For $1\leq q\leq n+1$, define the alternating prefix sums of its maximal principal Pfaffians by
  \[Z_q(A) = \sum_{i=1}^{q-1} (-1)^{i+1} \pf(A([n] \setminus \{i\})).\]
\begin{lemma}[Alternating Prefix Sums]
  \label{lem:alternating-prefix-sum}
  Let $n \geq 3$ be odd. 
  For independent commuting indeterminates $a=(a_J)_{J\in\cI_n}$, each $Z_q(A_n(a))$ is a polynomial with nonnegative integer coefficients. 
  Consequently, if $A$ is a real skew-symmetric matrix with nonnegative interval coefficients, then
  \[Z_q(A)=Z_q(A_n(a(A)))\geq 0.\]
\end{lemma}

\begin{proof}
  The case $q = 1$ is immediate because $Z_1(A_n(a)) = 0$.
  Fix $2 \leq q \leq n+1$ and write $A = A_n(a)$.
  Adjoin a zero row and column to $A$, indexed by $0$ and consider the skew-symmetric matrix 
  \[B(t)=(0\oplus A) +t(\tilde{\bone}_q\oplus 0_{n+1-q}),\]
  with rows and columns indexed by $\{0\} \cup [n]$.

  This is an interval matrix on $\{0, 1, \ldots , n\}$ with interval coefficients $a_J$ for $J \in \cI_n$ and additionally $t = a_{[0,q-1]}$. 
  All other interval coefficients involving $0$ are set to zero.
  Thus by \autoref{thm:coefficientwise-positive} for even-order interval matrices, $\pf(B(t))$ has nonnegative integer coefficients in $t$ and the original interval variables. 

  Expanding the Pfaffian along the $0$ row gives 
  \[ \pf(B(t)) =t\sum_{i=1}^{q-1}(-1)^{i+1} \pf\bigl(B(t)([n]\setminus\{i\})\bigr).\]
  Since $B(0)=0\oplus A$, the coefficient of $t$ in $\pf(B(t))$ is  
  \[\sum_{i=1}^{q-1}(-1)^{i+1} \pf\left(A_n(a)([n] \setminus \{i\})\right) = Z_q(A_n(a)).\]
  Hence $Z_q(A_n(a))$ is coefficientwise nonnegative. 
  Specializing the indeterminates to $a(A)$ gives $Z_q(A) = Z_q(A_n(a(A))) \geq 0$.
\end{proof}

Besides its role in proving the endpoint inequality, this coefficientwise positivity will also provide bounds for bordered Pfaffians.
We next show that the Pfaffian is bounded below by the product of its endpoint entry and the Pfaffian of the remaining principal submatrix. 
Iterating this inequality will yield the Rainbow Bound.

\begin{prop}[Endpoint Inequality]
  \label{prop:endpoint-inequality}
  Let $n \geq 2$ be even.
  For independent commuting indeterminates $a = (a_J)_{J \in \cI_n}$, the polynomial
  \[\pf_n(a) - (A_n(a))_{1,n} \pf(A_n(a)([n] \setminus \{1,n\})) \]
  has nonnegative integer coefficients.

  Consequently, if $A$ is a real $n \times n$ skew-symmetric matrix with nonnegative interval coefficients $a_J(A)\geq 0$ for all $J \in \cI_n$, then 
  \[\pf(A) \geq A_{1,n} \pf\left(A([n] \setminus \{1,n\})\right).\]
  More generally, for every $S \subseteq [n]$ with $|S| \geq 2$ even, 
  \[\pf\left(A(S)\right) \geq A_{\min(S), \max(S)} \pf\left(A(S\setminus \{\min(S), \max(S)\})\right).\]
\end{prop}

\begin{proof}
For $n = 2$, we have $\pf_2(a) = (A_2(a))_{1,2} \pf(A_2(a)(\emptyset)) = (A_2(a))_{1,2}$ by the definition of the Pfaffian for $2 \times 2$ matrices.
Observe that restriction to principal submatrices preserves the interval-matrix structure.
Let $S=\{s_1<\ldots<s_k\}\subseteq[n]$.
For each $I\in\cI_n$, let
\[I_S=\{r\in[k]:s_r\in I\}.\]
Then $I_S$ is an interval, possibly empty or a singleton. 
For $J\in\cI_k$, put
\[b_J=\sum_{\substack{I\in\cI_n\\I_S=J}}a_I.\]
After relabeling $S$ by $[k]$, we have
\[A_n(a)(S)=A_k(b).\]

Suppose $n \geq 4$ and let $B=A_n(a)([n]\setminus \{1\})$ indexed increasingly by $[n-1]$. 
By the preceding observation and \autoref{lem:alternating-prefix-sum}, every $Z_q(B)$ has nonnegative integer coefficients in the indeterminates $a_J$.

For $2 \leq m \leq n$ define the partial first row expansion of the Pfaffian 
\[T_m = \sum_{j = 2}^m (-1)^j (A_n(a))_{1,j} \pf \bigl(A_n(a)([n]\setminus \{1,j\}) \bigr). \]
Thus $T_n = \pf_n(a)$.
By the definition of $Z_q(B)$, \[T_m = \sum_{j=2}^m (A_n(a))_{1,j}(Z_j(B) - Z_{j-1}(B)). \]
Note that from the definition of $Z_q$, \[Z_j(B)-Z_{j-1}(B) = (-1)^j\pf\bigl(A_n(a)([n]\setminus\{1,j\})\bigr).\]
Since $Z_1(B)  = 0$, summation by parts then yields 
\[T_m = (A_n(a))_{1,m} Z_m(B) + \sum_{q=2}^{m-1}\bigl((A_n(a))_{1,q}-(A_n(a))_{1,q+1}\bigr)Z_q(B).\]

For $2 \leq q < n$, the definition of the interval matrix gives
\[(A_n(a))_{1,q}-(A_n(a))_{1,q+1}=a_{[1,q]}.\]
Taking $m=n-1$ in the preceding identity, we obtain
\[T_{n-1}=(A_n(a))_{1,n-1}Z_{n-1}(B)+\sum_{q=2}^{n-2}a_{[1,q]}Z_q(B).
\]
Since $(A_n(a))_{1,n-1}=a_{[1,n-1]}+a_{[1,n]}$ and each $Z_q(B)$ has nonnegative integer coefficients, $T_{n-1}$ has nonnegative integer coefficients.

On the other hand, since $n$ is even, the final term in the first-row expansion has positive sign, so
\[T_{n-1}=\pf_n(a)-(A_n(a))_{1,n}\pf(A_n(a)([n]\setminus\{1,n\})).
\]
This proves the coefficientwise assertion.
Specializing the indeterminates to the nonnegative interval coefficients of $A$ gives the first inequality.
Specializing $a$ to $a(A)$ gives $A(S) = A_k(b)$ with every $b_J \geq 0$. 
Evaluating the coefficientwise statement for $k$ at $b$ gives the inequality for $A(S)$.
\end{proof}

We now apply this result repeatedly to obtain the rainbow bound \autoref{thm:rainbow-bound}.

\begin{proof}[Proof of \autoref{thm:rainbow-bound}]
  Let $A$ be a real skew-symmetric matrix with nonnegative interval coefficients.
  For $1 \leq k \leq n/2$, let $S_k = [k, n+1-k]$ and $S_{n/2 + 1} = \emptyset$ and put $Q_k = \pf(A(S_k))$ so $Q_{n/2+1} = 1$.
  By \autoref{prop:endpoint-inequality} applied to $A(S_k)$, we have \[Q_k \geq A_{k, n+1-k} Q_{k+1}.\]
  Since the interval coefficients of $A$ are nonnegative, each $A_{k,n+1-k}\geq0$.
  Iterating the inequality and using $\pf(\emptyset)=1$ gives
  \[
    \pf(A)\geq
    \prod_{k=1}^{n/2}A_{k,n+1-k}.
  \]

  Moreover, by \autoref{lem:interval-decomposition},
  \[A_{k,n+1-k}\geq a_{[1,n]}(A)=A_{1,n},\]
  since the sum defining $A_{k,n+1-k}$ includes the nonnegative term $a_{[1,n]}(A)$.
  Hence
    \[\pf(A)\geq \prod_{k=1}^{n/2}A_{k,n+1-k} \geq A_{1,n}^{n/2}.\]

  Finally, by \autoref{nonnegative-interval-coefficient}, the matrix $A_f(x)$ has nonnegative interval coefficients whenever $f$ is positive, nonincreasing, and convex and $x_1 < x_2 < \ldots < x_n$. 
  Thus
  \[ \pf(A_f(x)) \geq \prod_{k=1}^{n/2} f(x_{n+1-k}-x_k), \]
  as claimed.
\end{proof}

\section{Bordered Pfaffians and Albouy--Moeckel}
The interval decomposition also yields positivity for Pfaffians obtained by adjoining one or two additional vectors. 
These results provide the remaining nonvanishing statements needed for the Albouy--Moeckel conjecture.

The first result concerns the singly bordered Pfaffian corresponding to $K_n^L(x)$ in Albouy--Moeckel.

\begin{cor}[Singly Bordered Pfaffian]
  \label{cor:single-border}
Let $n=2m+1\geq3$, and let $A$ be a real skew-symmetric $n\times n$ matrix with nonnegative interval coefficients. 
Writing
$L=(1,\ldots,1)^T$, we have
\[ \pf\begin{pmatrix}
0 & L^T\\
-L & A
\end{pmatrix} \geq A_{1,n}^m.\]
In particular, if $f$ is positive, nonincreasing, and convex, then
\[ \pf\begin{pmatrix}
0 & L^T\\
-L & A_f(x)
\end{pmatrix} \geq f(x_n-x_1)^m>0.
\]
\end{cor}

\begin{proof}
  Expanding along the top row, \[\pf\begin{pmatrix} 0 & L^T\\ -L & A \end{pmatrix}  = Z_{n+1}(A).\]
  Applying this identity to $A_n(a)$, \autoref{lem:alternating-prefix-sum} shows that $Z_{n+1}(A_n(a))$ is a polynomial with nonnegative integer coefficients in the interval variables.
  When only $a_{[1,n]} = c$ is nonzero, every maximal principal Pfaffian is $c^m$.
  Thus under this specialization \[Z_{n+1}(A_n(a)) =c^m\sum_{i=1}^{2m+1}(-1)^{i+1}=c^m.\]
  Thus the coefficient of $a_{[1,n]}^m$ in $Z_{n+1}(A_n(a))$ is $1$.
  Since all other coefficients are nonnegative, specializing to the nonnegative interval coefficients of $A$ gives
  \[ Z_{n+1}(A)\geq a_{[1,n]}(A)^m=A_{1,n}^m. \]
  The assertion for convex kernels follows from \autoref{nonnegative-interval-coefficient}.
\end{proof}

The second result concerns the doubly bordered Pfaffian corresponding to $K_n^{xL}(x)$ in Albouy--Moeckel.

\begin{cor}[Doubly Bordered Pfaffian]
  \label{cor:double-border}
Let $n=2m\geq2$, let $A$ be a real skew-symmetric $n\times n$ matrix with nonnegative interval coefficients, and let $x_1<\ldots<x_n$. 
Writing $L=(1,\ldots,1)^T$, we have
\[ \pf\begin{pmatrix}
0 & 0 & x^T\\
0 & 0 & L^T\\
-x & -L & A
\end{pmatrix} \geq A_{1,n}^{m-1} \sum_{k=1}^m(x_{2k}-x_{2k-1}).
\]
In particular, this Pfaffian is strictly positive when $A=A_f(x)$ for any positive, nonincreasing, convex kernel $f$.
\end{cor}

\begin{proof}
  Let $d_g = x_{g+1} - x_g > 0$ and $h_g = \mathbf{1}_{\{g+1, \ldots, n\}}$.
  Then \[x = x_1 L  + \sum_{g=1}^{n-1} d_g h_g\] gives
  \[ \pf\begin{pmatrix}
    0 & 0 & x^T\\
    0 & 0 & L^T\\
    -x & -L & A
    \end{pmatrix} = 
      \sum_{g=1}^{n-1} d_g D_g(A),
  \]
  where $D_g(A)$ is the Pfaffian of the doubly bordered matrix with $x$ replaced by $h_g$.
  Since the Pfaffian is alternating in its two bordering vectors, replacing $L$ by $L-h_g$ leaves $D_g(A)$ unchanged. But $L-h_g=\mathbf{1}_{\{1,\ldots,g\}}$.
  Now, consider the above matrix in interval variables as $D_g(A_n(a))$, and add in a $0^\text{th}$ and $n+1$ row and column with $a_{[0,g]} = s$ and $a_{[g+1, n+1]} = t$.
  Denote this new matrix by $B_g(s,t)$.
  Then $D_g(A_n(a))$ is the $st$ coefficient of $\pf(B_g(s,t))$.

  Every matching contributing to the coefficient of $st$ must pair vertex $0$ with some $i\leq g$, contributing one factor of $s$, and vertex $n+1$ with some $j>g$, contributing one factor of $t$. 
  Although the added intervals also modify entries within the original $n\times n$ block, only the terms independent of $s,t$ from these entries can contribute to the coefficient of $st$. 
  These terms are precisely the entries of $A_n(a)$.
  Moving vertex $n+1$ to the first position introduces the permutation sign $(-1)^{n+1}=-1$, since $n$ is even. 
  Its bordering entries simultaneously change sign, so the two signs cancel.
  Then by \autoref{thm:coefficientwise-positive} $D_g(A_n(a))$ has nonnegative integer coefficients in the interval variables; so we only need to compute one coefficient to get a lower bound.

  Returning to the original $(h_g,L)$ form of $D_g(A)$, expand along the first bordering row. 
  Under the specialization $a_{[1,n]}=c$ and all other variables zero, each remaining singly bordered principal Pfaffian has the all-ones border and equals $c^{m-1}$.
  So expanding along the first border row, we get 
    \[ D_g(A)=c^{m-1}\sum_{j=g+1}^{2m}(-1)^j =
    \begin{cases}
    c^{m-1}, & g\text{ odd},\\
    0, & g\text{ even}.
    \end{cases}\]
  Consequently,
    \[
    \sum_{g=1}^{n-1}d_g D_g(A) \geq A_{1,n}^{m-1}\sum_{k=1}^{m}d_{2k-1},
    \]
which is the desired inequality.
\end{proof}

With $f(r) = r^{-2}$, and $x=(x_1,\ldots,x_n)^T$ and $L=(1,\ldots,1)^T$ of appropriate size, the Pfaffians of Albouy and Moeckel, in their normalization, satisfy

\begin{align*}
  K_{2m}(x) & = (-1)^m m! \pf\left(A_f(x)\right)\\
  K^L_{2m+1}(x) & = (-1)^m m! \pf\begin{pmatrix}
0 & L^T\\
-L & A_f(x)
\end{pmatrix}\\
  K^{xL}_{2m}(x) &= (-1)^m (m-1)! \pf\begin{pmatrix}
    0 & 0 & x^T\\
    0 & 0 & L^T\\
    -x & -L & A_f(x)
    \end{pmatrix}
\end{align*}

\begin{cor}[Albouy--Moeckel]
  For the Newtonian kernel $f(r)=r^{-2}$, the ordinary Pfaffians in even order, the singly bordered Pfaffians in odd order, and the doubly bordered Pfaffians in even order are all nonzero for every strictly ordered configuration. 
  Consequently, the nonvanishing conjecture of Albouy and Moeckel \cite{albouyMoeckel2000} holds.
\end{cor}

\begin{proof}
  The function $f(r)=r^{-2}$ is positive, nonincreasing, and convex. 
  The ordinary Pfaffian is strictly positive by \autoref{thm:pfaffian-positive-diameter}, and both bordered Pfaffians are strictly positive by \autoref{cor:single-border} and \autoref{cor:double-border}.
\end{proof}

By Theorem 1 of Albouy and Moeckel \cite{albouyMoeckel2000} the nonvanishing of these Pfaffians provides that $\dim S(x) = 2$ where $S(x)$ is the space of real mass vectors making the configuration central.

\section{Optimality and Examples}
\label{sec:optimal-bound}

We now show that the rainbow matching is the only perfect matching whose associated product gives a universal lower bound. 
In fact, for any other matching, we construct positive, nonincreasing, convex kernels and point configurations for which the ratio of the Pfaffian to that matching product approaches zero.

\begin{lemma}[Affine Kernel Pfaffian]
  \label{lem:affine-kernel-pfaffian}
  Let $n \geq 2$ be even, $c \in \bbR$ and $x_1, \ldots, x_n \in \bbR$ not necessarily ordered. 
  Let $A$ be the skew-symmetric matrix with $A_{i,j} = c - (x_j - x_i)$ for $i <  j$.
  Then \[\pf(A) = c^{n/2-1} \left(c - \sum_{k=1}^{n/2} (x_{2k} - x_{2k-1})\right).\]
\end{lemma}

\begin{proof}
  Let $d_k = x_{2k} - x_{2k-1}$ for $1 \leq k \leq n/2$, and subtract row $2k$ from row $2k - 1$ and then column $2k$ from column $2k - 1$.
  This is multiplying on the left by $P = \begin{pmatrix}1 & -1 \\ 0 & 1 \end{pmatrix}^{\oplus \frac{n}{2}}$ and on the right by $P^T$.
  Since $\pf(PAP^T) = \det(P) \pf(A)$ and $\det(P) = 1$, this operation preserves the Pfaffian.
  Let $C = PAP^T$.
  Letting $\delta_{ij} = \begin{cases} 1 & i=j \\ 0 & i \neq j\end{cases}$ we can calculate 
  \begin{align*}
    C_{2k-1,2l-1} &= 0 \\
    C_{2k-1, 2l} &= c \delta_{kl} - d_k \qquad 1 \leq k,l\leq n/2
  \end{align*}
  Consequently, every nonzero matching pairs odd indices exclusively with even indices, so the even--even entries do not contribute to the Pfaffian.
  Let $H$ be the odd-even submatrix of $C$ from taking $H_{k,l} = C_{2k-1,2l}$.
  Reordering the odd-indexed vertices before the even-indexed vertices introduces the sign $(-1)^{\binom{n/2}{2}}$. 
  The Pfaffian-determinant identity for the resulting block matrix introduces the same sign, so they cancel, giving $\pf(A)=\pf(C)=\det(H)$.
  Note that $H = cI_{n/2} - d \mathbf{1}^T$ where $d = (d_1, \ldots, d_{n/2})^T$ is the vector of gaps within the pairs $\{2k-1, 2k\}$.
  Then by the matrix-determinant lemma, for $ c \neq 0$,
  \begin{align*}
     \det(cI_{n/2} - d \mathbf{1}^T) &= \left(1 - \mathbf{1}^T(cI_{n/2})^{-1}d\right)\det(cI_{n/2}) \\
    &= \left(1 - \frac{1}{c}\mathbf{1}^Td\right)c^{n/2} \\
    &= c^{n/2 - 1}\left(c - \sum_{k=1}^{n/2}d_k\right)
  \end{align*}
  For $c = 0$, the result holds as well, because this is a polynomial equation in $c$.
  Since $d_k = (x_{2k} - x_{2k-1})$, this is what we wanted to prove.
\end{proof}

\begin{lemma}[Constant Outer Layer]
  \label{lem:constant-outer-layer}
  Let $n \geq 4$ be even. 
  If $A$ is an $n \times n$ skew-symmetric matrix such that, for some $c \in \bbR$, $A_{1,j} = c$ for all $2 \leq j \leq n$ and $A_{i,n} = c$ for all $1 \leq i < n$, then 
  \[\pf(A) = c \pf\left(A([2,n-1])\right).\]
\end{lemma}

\begin{proof}
  Let $E_{i,j}$ be the matrix with $1$ in the $i,j$ position and $0$ elsewhere.
  Set $P = I_n + E_{1,n}$ and $B = P^T A P$; $B$ is $A$ with row $1$ added to row $n$ and similarly for columns $1$ and $n$.
  As in the previous lemma, $\det P = 1$, so $\pf(A) = \pf(B)$.
  Then $B_{i,n} = 0$ for $i \neq 1$ and $B_{1,n} = c$.
  Also $B([2,n-1]) = A([2,n-1])$ because only the first and last rows and columns were modified.
  Expanding $\pf(B)$ along the last column, the only surviving matchings involving $n$ must have the $\{1,n\}$ edge, so $\pf(A) = c \pf\left(A([2,n-1])\right)$.
\end{proof}

Consider the case of $n = 4$.

\begin{example}
  Let $T > 0$ and set $x = (0, T, T+1, 2T + 1)$ and $f(r) = 1 + (2T + 1 - r)_{+}$ where $(x)_+ = \max\{x,0\}$.
  Then $f$ is positive, nonincreasing, and convex.
  We have that 
  \[A_f(x) = \begin{pmatrix}
    0 & T + 2 & T + 1 & 1 \\ 
    -(T + 2) & 0 & 2T + 1 & T + 1 \\
    -(T + 1) & -(2T + 1) & 0 & T + 2 \\
    - 1 & -(T + 1) & -(T + 2) & 0
  \end{pmatrix}\]
  Hence $\pf(A_f(x)) = 4(T+1)$.
  Our rainbow bound gives $1 \cdot (2T + 1) < 4(T + 1)$, but the adjacent matching gives $(T+2)^2$ and the crossing matching gives $(T+1)^2$, both of which exceed $\pf(A_f(x))$ for $T > 3$.
  The diameter bound from \autoref{thm:pfaffian-positive-diameter} only provides a lower bound of $1$, which is much weaker.
\end{example}

This shows that for $n = 4$, only the rainbow matching bound survives.
We now generalize this construction to all even $n$.

\begin{proof}[Proof of \autoref{thm:rainbow-bound-optimal}]
Let $k$ be the number of consecutive outer rainbow pairs \[ \{1,n\}, \{2,n-1\},\ldots, \{k, n+1 - k\}\] contained in $M$ stopping at the first missing pair.
Put $N = n - 2k$ and $m = N/2$.
Since $M \neq R$, we have $N \geq 4$ because $k \leq n/2 - 2$; if $k = n/2 - 1$, then the remaining vertices are forced to pairs as $\{n/2, n/2+1\}$ so $M = R$.
Relabel (preserving order) the matching induced on $[k+1, n-k]$ as $M_0 \in \Pi_N$.
Then $\{1,N\} \not \in M_0$ by construction.

For $T > 0$, take the $N$-tuple 
\[ y = (0, T, T+1, \ldots T+N - 3, 2T+ N - 3)\]
Note that $y_1 = 0$, $y_i = T + (i-2)$ for $i \in [2, N-1]$, and $y_N  = 2T + N - 3 = D$.
Define the positive, nonincreasing, convex function \[f_T(r) = 1 + (D-r)_+\] where $(x)_+ = \max\{x,0\}$.
All the pairwise distances among the $y_i$ are at most $D = y_N - y_1$, so for $i < j$ we have $f_T(y_j - y_i) = c-(y_j - y_i)$ where $c = D+1 = 2T +N - 2$.
Then the gaps within pairs of consecutive points sum to $2T + m-2 = \sum_{j=1}^m \left(y_{2j} - y_{2j-1}\right)$ hence by \autoref{lem:affine-kernel-pfaffian} we have 
\[\pf(A_{f_T}(y)) = c^{m - 1}(c - (2T + m - 2)) = m(2T + N - 2)^{m-1}.\]

We know $\{1,N\}$ is not in the matching $M_0$ and 
\[f_T(y_j - y_1), f_T(y_N - y_i) \geq T+1 \text{ for all  } i,j \in [2,N-1].\] 
Similarly, $f_T(y_j - y_i) \geq (2T+1)$ for all $i < j \in [2,N-1]$.
This means the product associated with $M_0$ is at least $(T+1)^2 (2T+1)^{m-2}$.
Note that \[\lim_{T \to \infty} \frac{m(2T + N - 2)^{m-1}}{(T+1)^2 (2T+1)^{m-2}} = 0.\]

To transport this result back to $M$, let $x_{k+i} = y_i$ and for $1 \leq i \leq k$, set
\[x_{k+1-i} = y_1 - 2iD \qquad x_{n-k+i} = y_N + 2iD.\]
Every pairwise distance $r$ involving one of the new points is at least $2D$, so $f_T(r) = 1$.
Repeatedly applying \autoref{lem:constant-outer-layer} shows that 
\[\pf(A_{f_T}(x)) = \pf(A_{f_T}(y)).\]
Additionally, the product associated with $M$ is the same as the product associated with $M_0$ as all new edges of the pairing contribute exactly $1$.
Since the above ratio tends to $0$ we have for $f$ positive, nonincreasing, convex and $x_1 < \ldots < x_n$
  \[
    \inf_{f, x}
    \frac{\pf(A_f(x))} {\displaystyle\prod_{\substack{\{i,j\}\in M\\i<j}} f(x_j-x_i)} = 0.
  \]
\end{proof}

\section{Further Directions}
The optimality of the rainbow matching is specific to the full class of positive, nonincreasing, convex kernels. 
For log-convex kernels, the four-point Pfaffian bounds the products associated with all three perfect matchings. 
This suggests investigating whether every perfect-matching product is bounded above by the Pfaffian in arbitrary even order.
The four-point case follows directly from log-convexity, monotonicity, and the rainbow bound.
\begin{rmk}
Let $s = x_2 - x_1$, $t = x_3 - x_2$, and $r = x_4 - x_3$.
Log-convexity gives 
\[
\log f(t)+\log f(s+t+r) \geq \log f(s+t)+\log f(t+r),
\]
So $A_f(x)_{14}A_f(x)_{23} \geq A_f(x)_{13} A_f(x)_{24}$.
So $\pf(A_f(x)) \geq A_f(x)_{12}A_f(x)_{34}$, the rainbow bound gives $\pf(A_f(x)) \geq A_f(x)_{14}A_f(x)_{23}$ since log-convex implies convex, and monotonicity gives \[\pf(A_f(x)) \geq A_f(x)_{12}A_f(x)_{34} \geq A_f(x)_{13}A_f(x)_{24}.\]
Hence $\pf(A_f(x))$ exceeds all three matching bounds.
\end{rmk}

\begin{conj}
Let $f:(0,\infty)\to(0,\infty)$ be positive, nonincreasing, and log-convex, and $x_1 < x_2 < \ldots < x_n$ with $n$ even, then for every perfect matching $M$
  \[ \pf(A_f(x)) \geq \prod_{\substack{\{i,j\} \in M \\ i < j}} f(x_j-x_i).\]
\end{conj}

Since for $i<j<k<l$, monotonicity gives
\[A_f(x)_{ij}A_f(x)_{kl}\geq A_f(x)_{ik}A_f(x)_{jl},\]
 uncrossing edges only increases the bound.
Each uncrossing strictly decreases the total length of the matching edges. 
Repeated uncrossing eventually produces a noncrossing matching without decreasing its product.
Therefore repeated uncrossing shows that it suffices to establish the conjecture for noncrossing perfect matchings, which are enumerated by Catalan numbers.

\section*{Acknowledgments}
AI systems (Claude Opus 5.5 and GPT-6 via ChatGPT) were used extensively in the development of this work, including the generation of the original mathematical proofs and the Lean 4 formalization. The author subsequently checked and revised the mathematical arguments, reviewed the definitions and theorem statements of the formalization for correspondence with the intended results, and wrote the manuscript. The Lean proofs are machine-checked in Lean 4. The author takes full responsibility for the correctness of the mathematical results and the contents of this paper.

This research received no external funding.

\section*{Code Availability}

Part of this paper has been formalized in the Lean~4 proof assistant~\cite{lean4}, using Mathlib~\cite{mathlib}.
The formalization is available in the following repository~\cite{pfaffianPositivityLean}:
\begin{center}
\url{https://github.com/wlpjohnson-math/convex-pfaffian-positivity-lean}
\end{center}
The version described here is release \texttt{v1.0.0}, under the Apache License, Version~2.0, and builds with Lean~v4.34.1 and Mathlib~v4.34.1.

The formalization includes the Interval Decomposition and Twin Rows lemmas, Coefficient Positivity, the nonnegativity of interval coefficients for convex kernels, the Pfaffian Diameter Bound, and the Rainbow Bound. 
It also proves the real-valued forms of the Alternating Prefix Sums lemma and the Endpoint Inequality for $S=[n]$.
For optimality, it establishes that no other single-matching product is a universal lower bound, but not the stronger infimum-zero statement of \autoref{thm:rainbow-bound-optimal}.
The bordered Pfaffian results and the full Albouy--Moeckel corollary are not formalized.

The Lean development defines the Pfaffian by the signed perfect-matching expansion, with sign $(-1)^{\crossings(M)}$, and uses families indexed by $\mathbb N$ restricted to finite sets. 
Only entries above the diagonal contribute. 
The repository's \texttt{docs/PaperCorrespondence.md} identifies the Lean theorems corresponding to the paper's results and explains the differences in conventions.
The file \texttt{PfaffianLean/Spec.lean} provides reviewed statements and guards for the definitions and axioms used, allowing the formalized assertions to be inspected without reading their proofs.

\bibliographystyle{amsalpha}
\bibliography{references}

\end{document}